\documentclass{amsart}
\usepackage{setspace}
\usepackage{a4}
\usepackage{amsthm}
\usepackage{latexsym}
\usepackage{amsfonts}
\usepackage{graphicx}
\usepackage{textcomp}
\usepackage{cite}
\usepackage{enumerate}
\usepackage{amssymb}
\usepackage{hyperref}
\usepackage{amsmath}
\usepackage{tikz}
\usepackage[mathscr]{euscript}
\usepackage{mathtools}
\newtheorem{theorem}{Theorem}[section]

\newtheorem{conjecture}[theorem]{Conjecture}
\newtheorem{corollary}[theorem] {Corollary}
\newtheorem{definition}[theorem]{Definition}

\newtheorem{question}[theorem]{Question}
\title{This is the title}
\usepackage{fancyhdr}

\begin{document}
\hrule\hrule\hrule\hrule\hrule
\vspace{0.3cm}	
\begin{center}
{\bf{CONTINUOUS DEUTSCH UNCERTAINTY PRINCIPLE AND CONTINUOUS KRAUS CONJECTURE}}\\
\vspace{0.3cm}
\hrule\hrule\hrule\hrule\hrule
\vspace{0.3cm}
\textbf{K. MAHESH KRISHNA}\\
Post Doctoral Fellow \\
Statistics and Mathematics Unit\\
Indian Statistical Institute, Bangalore Centre\\
Karnataka 560 059, India\\
Email: kmaheshak@gmail.com\\

Date: \today
\end{center}

\hrule\hrule
\vspace{0.5cm}
\textbf{Abstract}: Let $(\Omega, \mu)$,  $(\Delta, \nu)$ be   measure spaces  and $\{\tau_\alpha\}_{\alpha\in \Omega}$, $\{\omega_\beta\}_{\beta \in \Delta}$ be  	1-bounded continuous Parseval frames for a  Hilbert  space $\mathcal{H}$.
Then   we show that 
\begin{align}\label{UE}
\log (\mu(\Omega)\nu(\Delta))\geq S_\tau(h)+S_\omega (h)\geq -2  \log \left(\frac{1+\displaystyle \sup_{\alpha \in \Omega, \beta \in \Delta}|\langle\tau_\alpha , \omega_\beta\rangle|}{2}\right)	, \quad \forall h \in \mathcal{H}_\tau \cap \mathcal{H}_\omega,
\end{align}
where 
\begin{align*}
	&\mathcal{H}_\tau := \{h_1 \in \mathcal{H}: \langle h_1 , \tau_\alpha \rangle \neq 0, \alpha \in \Omega\}, \quad \mathcal{H}_\omega := \{h_2 \in \mathcal{H}: \langle h_2, \omega_\beta \rangle \neq 0, \beta \in \Delta\},\\
&S_\tau(h):= -\displaystyle\int\limits_{\Omega}\left|\left \langle \frac{h}{\|h\|}, \tau_\alpha\right\rangle \right|^2\log \left|\left \langle \frac{h}{\|h\|}, \tau_\alpha\right\rangle \right|^2\,d\mu(\alpha), \quad  \forall h \in \mathcal{H}_\tau, \\
& S_\omega (h):= -\displaystyle\int\limits_{\Delta}\left|\left \langle \frac{h}{\|h\|}, \omega_\beta\right\rangle \right|^2\log \left|\left \langle \frac{h}{\|h\|}, \omega_\beta\right\rangle \right|^2\,d\nu(\beta), \quad  \forall h \in \mathcal{H}_\omega.
\end{align*}
 We call Inequality (\ref{UE}) as \textbf{Continuous Deutsch Uncertainty Principle}. Inequality (\ref{UE}) improves  the uncertainty principle  obtained by Deutsch \textit{[Phys. Rev. Lett., 1983]}. We formulate Kraus conjecture for 	1-bounded continuous Parseval frames. We also derive continuous Deutsch uncertainty principles for Banach spaces.



\hrule

\tableofcontents
\hrule
\section{Introduction}
Let $\mathcal{H}$ be a finite dimensional Hilbert space. Given an orthonormal basis  $\{\omega_j\}_{j=1}^n$ for $\mathcal{H}$, the \textbf{(finite) Shannon entropy}  at a point $h \in \mathcal{H}_\tau$ is defined as 
\begin{align*}
	S_\tau (h)\coloneqq - \sum_{j=1}^{n} \left|\left \langle \frac{h}{\|h\|}, \tau_j\right\rangle \right|^2\log \left|\left \langle \frac{h}{\|h\|}, \tau_j\right\rangle \right|^2\geq 0, 
\end{align*}
where $\mathcal{H}_\tau \coloneqq \{h \in \mathcal{H}: \langle h , \tau_j \rangle\neq 0, 1\leq j \leq n \}$. In 1983, Deutsch derived following uncertainty principle for Shannon entropy  \cite{DEUTSCH}.
\begin{theorem} (\textbf{Deutsch Uncertainty Principle}) \cite{DEUTSCH} \label{DU}
Let $\{\tau_j\}_{j=1}^n$,  $\{\omega_j\}_{j=1}^n$ be two orthonormal bases for a  finite dimensional Hilbert space $\mathcal{H}$. Then 
	\begin{align}\label{DUP}
2 \log n \geq S_\tau (h)+S_\omega (h)\geq -2 \log \left(\frac{1+\displaystyle \max_{1\leq j, k \leq n}|\langle\tau_j , \omega_k\rangle|}{2}\right)	\geq 0, \quad \forall h \in \mathcal{H}_\tau \cap \mathcal{H}_\omega.
	\end{align}
\end{theorem}
In 1988, followed by a conjecture of Kraus \cite{KRAUS} made in 1987, Maassen and Uffink improved Inequality (\ref{DUP}) \cite{MAASSENUFFINK}.
\begin{theorem} (\textbf{Kraus Conjecture/Maassen-Uffink  Uncertainty Principle}) \cite{KRAUS, MAASSENUFFINK} \label{KMU}
	Let $\{\tau_j\}_{j=1}^n$,  $\{\omega_j\}_{j=1}^n$ be two orthonormal bases for a  finite dimensional Hilbert space $\mathcal{H}$. Then 
	\begin{align*}
2 \log n \geq S_\tau (h)+S_\omega (h)\geq -2 \log \left(\displaystyle\max_{1\leq j, k \leq n}|\langle\tau_j , \omega_k\rangle|\right)\geq 0, \quad \forall h \in \mathcal{H}_\tau \cap \mathcal{H}_\omega.
	\end{align*}
\end{theorem}
In 2013, Ricaud and Torr\'{e}sani \cite{RICAUDTORRESANI} showed that    Theorem \ref{KMU} holds for Parseval frames. 
\begin{theorem}(\textbf{Maassen-Uffink-Ricaud-Torr\'{e}sani Uncertainty Principle})\label{MU} \cite{RICAUDTORRESANI} Let $\{\tau_j\}_{j=1}^n$,  $\{\omega_j\}_{j=1}^n$ be two Parseval  frames for a  finite dimensional Hilbert space $\mathcal{H}$. Then 
	\begin{align*}
		2 \log n \geq S_\tau (h)+S_\omega (h)\geq -2 \log \left(\displaystyle\max_{1\leq j, k \leq n}|\langle\tau_j , \omega_k\rangle|\right)\geq 0, \quad \forall h \in \mathcal{H}_\tau \cap \mathcal{H}_\omega.
	\end{align*}
	\end{theorem}
Recently,  Banach space versions of Deutsch  uncertainty principle have been derived in \cite{KRISHNA}. To formulate them, we need some notions. Given a Parseval p-frame $\{f_j\}_{j=1}^n$   for $\mathcal{X}$, we define the \textbf{(finite) p-Shannon entropy} at a point $x \in \mathcal{X}_f$ as 

\begin{align*}
	S_f(x)\coloneqq -\sum_{j=1}^{n}\left|f_j\left(\frac{x}{\|x\|}\right)\right|^p\log \left|f_j\left(\frac{x}{\|x\|}\right)\right|^p\geq 0, 
\end{align*}
where $\mathcal{X}_f\coloneqq \{x \in \mathcal{X}:f_j(x)\neq 0, 1\leq j \leq n\}$. On the other way, given a Parseval p-frame $\{\tau_j\}_{j=1}^n$   for $\mathcal{X}^*$, we define the  \textbf{(finite) p-Shannon entropy} at a point $f \in \mathcal{X}^* _\tau$ as 
\begin{align*}
	S_\tau(f)\coloneqq -\sum_{j=1}^{n}\left|\frac{f(\tau_j)}{\|f\|}\right|^p\log \left|\frac{f(\tau_j)}{\|f\|}\right|^p\geq 0,
\end{align*}
where $\mathcal{X}^*_\tau\coloneqq \{f \in \mathcal{X}^*:f(\tau_j)\neq 0, 1\leq j \leq n\}$.
\begin{theorem}\label{M1}\cite{KRISHNA}
	 (\textbf{Functional Deutsch Uncertainty Principle}) Let $\{f_j\}_{j=1}^n$ and $\{g_k\}_{k=1}^m$ be  Parseval p-frames  for a finite dimensional Banach space $\mathcal{X}$. Then 
	\begin{align*}
		\frac{1}{(nm)^\frac{1}{p}}	\leq \displaystyle\sup_{y \in \mathcal{X}, \|y\|=1}\left(\max_{1\leq j\leq n, 1\leq k\leq m}|f_j(y)g_k(y)|\right)
	\end{align*}
	and 
	\begin{align*}
		\log (nm)\geq S_f (x)+S_g (x)\geq -p \log \left(\displaystyle\sup_{y \in \mathcal{X}_f\cap \mathcal{X}_g, \|y\|=1}\left(\max_{1\leq j\leq n, 1\leq k\leq m}|f_j(y)g_k(y)|\right)\right)> 0, \quad \forall x \in \mathcal{X}_f \cap \mathcal{X}_g.
	\end{align*}
\end{theorem}
\begin{theorem}\label{M2}
	(\textbf{Functional Deutsch Uncertainty Principle}) Let $\{\tau_j\}_{j=1}^n$ and $\{\omega_k\}_{k=1}^m$ be  two Parseval p-frames  for the dual $\mathcal{X}^*$ of a finite dimensional Banach space $\mathcal{X}$. Then 
	\begin{align*}
		\frac{1}{(nm)^\frac{1}{p}}	\leq \displaystyle\sup_{g \in \mathcal{X}^*, \|g\|=1}\left(\max_{1\leq j\leq n, 1\leq k\leq m}|g(\tau_j)g(\omega_k)|\right)
	\end{align*}
	and 
	\begin{align*}
		\log (nm) \geq	S_\tau (f)+S_\omega (f)\geq -p \log \left(\displaystyle\sup_{g \in \mathcal{X}^*_\tau \cap \mathcal{X}^*_\omega, \|g\|=1}\left(\max_{1\leq j\leq n, 1\leq k\leq m}|g(\tau_j)g(\omega_k)|\right)\right)> 0, \quad \forall f \in \mathcal{X}^*_\tau \cap \mathcal{X}^*_\omega.
	\end{align*}
\end{theorem}
In this paper,  we derive continuous versions of Theorem \ref{DU}, Theorem \ref{M1} and Theorem \ref{M2}. We also formulate a conjecture based on Theorem \ref{KMU}. We wish to say that functional continuous uncertainty principles are derived in \cite{KRISHNA2}.

\section{Continuous Deutsch  Uncertainty Principle and Continuous Kraus Conjecture }
In the paper,   $\mathbb{K}$ denotes $\mathbb{C}$ or $\mathbb{R}$ and $\mathcal{H}$ (resp. $\mathcal{X}$) denotes a Hilbert space (resp. Banach space) (need not be finite dimensional) over $\mathbb{K}$. 
We use  $(\Omega, \mu)$ to denote  a measure space. 
Continuous frames are introduced independently by Ali, Antoine and Gazeau \cite{ALIANTOINEGAZEAU} and Kaiser \cite{KAISER}. In the paper,   $\mathbb{K}$ denotes $\mathbb{C}$ or $\mathbb{R}$ and $\mathcal{H}$ denotes a finite dimensional Hilbert space.
\begin{definition}\cite{ALIANTOINEGAZEAU, KAISER}
	Let 	$(\Omega, \mu)$ be a measure space. A collection   $\{\tau_\alpha\}_{\alpha\in \Omega}$ in 	a  Hilbert  space $\mathcal{H}$ is said to be a \textbf{continuous Parseval frame}  for $\mathcal{H}$ if the following conditions hold.
	\begin{enumerate}[\upshape(i)]
		\item For each $h \in \mathcal{H}$, the map $\Omega \ni \alpha \mapsto \langle h, \tau_\alpha \rangle \in \mathbb{K}$ is measurable.
		\item 
		\begin{align*}
			\|h\|^2=\int\limits_{\Omega}|\langle h, \tau_\alpha \rangle|^2\,d\mu(\alpha), \quad \forall h \in \mathcal{H}.
		\end{align*}
	\end{enumerate}
\end{definition}
We consider the following subclass of continuous Parseval frames.
\begin{definition}\label{1B}
A continuous Parseval frame $\{\tau_\alpha\}_{\alpha\in \Omega}$   for $\mathcal{H}$ 	is said to be \textbf{1-bounded} if 
\begin{align*}
\|\tau_\alpha\|\leq 1, \quad \forall \alpha \in \Omega.
\end{align*}
\end{definition}
Note that if $\{\tau_j\}_{j=1}^n$ is a Parseval frame for a Hilbert space $\mathcal{H}$,  then $\|\tau_j\|\leq 1$, $ \forall 1\leq j \leq n$ (see Remark 3.12 in \cite{HANKORNELSONLARSONWEBER}). We are  unable to derive this for continuous frames. 
 Given a continuous 1-bounded Parseval frame   $\{\tau_\alpha\}_{\alpha\in \Omega}$ for $\mathcal{H}$,   we define the \textbf{continuous Shannon entropy} at a point $h \in \mathcal{H}_\tau$ as 
\begin{align*}
	S_\tau(h)\coloneqq  -\int\limits_{\Omega}\left|\left \langle \frac{h}{\|h\|}, \tau_\alpha\right\rangle \right|^2\log \left|\left \langle \frac{h}{\|h\|}, \tau_\alpha\right\rangle \right|^2\,d\mu(\alpha)\geq 0,
\end{align*}
where $\mathcal{H}_\tau \coloneqq \{h \in \mathcal{H}: \langle h , \tau_\alpha \rangle \neq 0, \alpha \in \Omega\}$. 
Following is the first fundamental  result of this paper. 
\begin{theorem}(\textbf{Continuous Deutsch Uncertainty Principle})\label{CDUP} Let $(\Omega, \mu)$,  $(\Delta, \nu)$ be   measure spaces  and $\{\tau_\alpha\}_{\alpha\in \Omega}$, $\{\omega_\beta\}_{\beta \in \Delta}$ be  	1-bounded continuous Parseval frames for a  Hilbert  space $\mathcal{H}$.
	Then   
	\begin{align}\label{CDSP}
	\log (\mu(\Omega)\nu(\Delta))\geq	S_\tau(h)+S_\omega (h)\geq -2  \log \left(\frac{1+\displaystyle \sup_{\alpha \in \Omega, \beta \in \Delta}|\langle\tau_\alpha , \omega_\beta\rangle|}{2}\right)	\geq 0, \quad \forall h \in \mathcal{H}_\tau \cap \mathcal{H}_\omega.
	\end{align}
\end{theorem}
\begin{proof}
	Since 	$1=\int\limits_{\Omega}\left|\left \langle \frac{h}{\|h\|}, \tau_\alpha\right\rangle \right|^2\,d\mu(\alpha)$ for all $h \in \mathcal{H} \setminus \{0\}$, $1=\int\limits_{\Delta}\left|\left \langle \frac{h}{\|h\|}, \omega_\beta\right\rangle \right|^2\,d\nu(\beta)$ for all $h \in \mathcal{H} \setminus \{0\}$ and  $\log$ is concave, using Jensen's inequality (cf. \cite{GARLING}) we get 
	\begin{align*}
	S_\tau(h)+S_\omega (h)&=\int \limits_{\Omega}\left|\left \langle \frac{h}{\|h\|}, \tau_\alpha\right\rangle \right|^2\log \left(\frac{1}{\left|\left \langle \frac{h}{\|h\|}, \tau_\alpha\right\rangle \right|^2}\right)\,d\mu(\alpha)+\int \limits_{\Delta}\left|\left \langle \frac{h}{\|h\|}, \omega_\beta\right\rangle \right|^2\log \left(\frac{1}{\left|\left \langle \frac{h}{\|h\|}, \omega_\beta\right\rangle \right|^2}\right)\,d\nu(\beta)\\
	&\leq \log \left(\int \limits_{\Omega}\left|\left \langle \frac{h}{\|h\|}, \tau_\alpha\right\rangle \right|^2\frac{1}{\left|\left \langle \frac{h}{\|h\|}, \tau_\alpha\right\rangle \right|^2}\,d\mu(\alpha)\right)+\log \left(\int \limits_{\Delta}\left|\left \langle \frac{h}{\|h\|}, \omega_\beta\right\rangle \right|^2\frac{1}{\left|\left \langle \frac{h}{\|h\|}, \omega_\beta\right\rangle \right|^2}\,d\nu(\beta)\right)\\
	&=\log (\mu(\Omega))+\log\nu (\Delta))=\log (\mu(\Omega)\nu (\Delta)), \quad \forall h \in \mathcal{H}_\tau \cap \mathcal{H}_\omega.
	\end{align*}
Let $h \in \mathcal{H}_\tau \cap \mathcal{H}_\omega$. Then using Buzano inequality \cite{BUZANO, FFUJIIKUBO} we get 
\begin{align*}
S_\tau(h)+S_\omega (h)&=-\int\limits_{\Delta}\int\limits_{\Omega}\left|\left \langle \frac{h}{\|h\|}, \tau_\alpha\right\rangle \right|^2\left|\left \langle \frac{h}{\|h\|}, \omega_\beta\right\rangle \right|^2\left[\log \left|\left \langle \frac{h}{\|h\|}, \tau_\alpha\right\rangle \right|^2+\log \left|\left \langle \frac{h}{\|h\|}, \omega_\beta\right\rangle \right|^2\right]\,d\mu(\alpha)\,d\nu(\beta)\\
&=-\int\limits_{\Delta}\int\limits_{\Omega}\left|\left \langle \frac{h}{\|h\|}, \tau_\alpha\right\rangle \right|^2\left|\left \langle \frac{h}{\|h\|}, \omega_\beta\right\rangle \right|^2\log \left|\left \langle \frac{h}{\|h\|}, \tau_\alpha\right\rangle \left \langle \frac{h}{\|h\|}, \omega_\beta\right\rangle\right|^2\,d\mu(\alpha)\,d\nu(\beta)\\
&=-2\int\limits_{\Delta}\int\limits_{\Omega}\left|\left \langle \frac{h}{\|h\|}, \tau_\alpha\right\rangle \right|^2\left|\left \langle \frac{h}{\|h\|}, \omega_\beta\right\rangle \right|^2\log \left|\left \langle \frac{h}{\|h\|}, \tau_\alpha\right\rangle \left \langle \frac{h}{\|h\|}, \omega_\beta\right\rangle\right|\,d\mu(\alpha)\,d\nu(\beta)\\
&\geq -2\int\limits_{\Delta}\int\limits_{\Omega}\left|\left \langle \frac{h}{\|h\|}, \tau_\alpha\right\rangle \right|^2\left|\left \langle \frac{h}{\|h\|}, \omega_\beta\right\rangle \right|^2\log \left(\left\|\frac{h}{\|h\|}\right\|^2\frac{\|\tau_\alpha\|\|\omega_\beta\|+|\langle\tau_\alpha, \omega_\beta \rangle|}{2}\right)	\,d\mu(\alpha)\,d\nu(\beta)\\
& =-2\int\limits_{\Delta}\int\limits_{\Omega}\left|\left \langle \frac{h}{\|h\|}, \tau_\alpha\right\rangle \right|^2\left|\left \langle \frac{h}{\|h\|}, \omega_\beta\right\rangle \right|^2\log \left(\frac{1+|\langle\tau_\alpha, \omega_\beta \rangle|}{2}\right)	\,d\mu(\alpha)\,d\nu(\beta)\\
&\geq -2\int\limits_{\Delta}\int\limits_{\Omega}\left|\left \langle \frac{h}{\|h\|}, \tau_\alpha\right\rangle \right|^2\left|\left \langle \frac{h}{\|h\|}, \omega_\beta\right\rangle \right|^2\log \left(\frac{1+\displaystyle \sup_{\alpha \in \Omega, \beta \in \Delta}|\langle\tau_\alpha, \omega_\beta \rangle|}{2}\right)	\,d\mu(\alpha)\,d\nu(\beta)\\
&=-2\log \left(\frac{1+\displaystyle \sup_{\alpha \in \Omega, \beta \in \Delta}|\langle\tau_\alpha, \omega_\beta \rangle|}{2}\right)\int\limits_{\Delta}\int\limits_{\Omega}\left|\left \langle \frac{h}{\|h\|}, \tau_\alpha\right\rangle \right|^2\left|\left \langle \frac{h}{\|h\|}, \omega_\beta\right\rangle \right|^2\,d\mu(\alpha)\,d\nu(\beta)\\
&\geq-2\log \left(\frac{1+\displaystyle \sup_{\alpha \in \Omega, \beta \in \Delta}|\langle\tau_\alpha, \omega_\beta \rangle|}{2}\right).
\end{align*}	
\end{proof}
Theorem \ref{CDUP} promotes following question.
\begin{question}
	Let $(\Omega, \mu)$,  $(\Delta, \nu)$ be   measure spaces,    $\mathcal{H}$ be a Hilbert space. For which pairs of 1-bounded continuous Parseval  frames $ \{\tau_\alpha\}_{\alpha\in \Omega}$  and   $\{\omega_\beta\}_{\beta\in \Delta}$ for  $\mathcal{H}$, we have equality in Inequality (\ref{CDSP})?
\end{question}
Based on Theorems \ref{MU} and Theorem \ref{CDUP} we formulate following conjecture.
\begin{conjecture} \textbf{(Continuous Kraus Conjecture)}
	\textbf{Let $(\Omega, \mu)$,  $(\Delta, \nu)$ be   measure spaces  and $\{\tau_\alpha\}_{\alpha\in \Omega}$, $\{\omega_\beta\}_{\beta \in \Delta}$ be  	1-bounded continuous Parseval frames    for a  Hilbert  space $\mathcal{H}$. Then 
	\begin{align*}
		S_\tau(h)+S_\omega (h)\geq -2  \log \left(\displaystyle \sup_{\alpha \in \Omega, \beta \in \Delta}|\langle\tau_\alpha , \omega_\beta\rangle|\right)	\geq 0, \quad \forall h \in \mathcal{H}_\tau \cap \mathcal{H}_\omega.	
	\end{align*}}
\end{conjecture}
Next we derive continuous Deutsch uncertainty for Banach spaces. We need a definition.
\begin{definition}\cite{FAROUGHIOSGOOEI}\label{A}
		Let 	$(\Omega, \mu)$ be a measure space and $\mathcal{X}$  be a   Banach space over $\mathbb{K}$. A collection $\{f_\alpha\}_{\alpha\in \Omega}$ in  $\mathcal{X}^*$  is said to be a \textbf{continuous Parseval p-frame} ($1\leq p <\infty$) for $\mathcal{X}$  if the following conditions  hold.
	\begin{enumerate}[\upshape(i)]
		\item For each $x \in \mathcal{X}$, the map $\Omega \ni \alpha \mapsto f_\alpha (x)\in \mathbb{K}$ is measurable.
		\item 
		\begin{align*}
			\|x\|^p=\displaystyle\int\limits_{\Omega}|f_\alpha (x)|^p\,d\mu(\alpha), \quad \forall x \in \mathcal{X}.
		\end{align*}
	\end{enumerate}
	If $\|\tau_\alpha\|\leq 1$, $\forall \alpha \in \Omega$, then we say that the frame $\{f_\alpha\}_{\alpha\in \Omega}$ is 1-bounded. 
\end{definition}
Similar to Definition \ref{1B}, we set the following.
\begin{definition}
	A continuous Parseval p-frame $\{f_\alpha\}_{\alpha\in \Omega}$   for $\mathcal{X}$ 	is said to be \textbf{1-bounded} if 
	\begin{align*}
		\|f_\alpha\|\leq 1, \quad \forall \alpha \in \Omega.
	\end{align*}
\end{definition}
Given a 1-bounded continuous Parseval p-frame $\{f_\alpha\}_{\alpha\in \Omega}$ for $\mathcal{X}$,   we define the \textbf{continuous p-Shannon entropy} at a point $x \in \mathcal{X}_f $ as 
\begin{align*}
	S_f(x)\coloneqq  -\int\limits_{\Omega}\left|f_\alpha \left(\frac{x}{\|x\|}\right) \right|^p\log\left|f_\alpha \left(\frac{x}{\|x\|}\right) \right|^p\,d\mu(\alpha)\geq 0,
\end{align*}
where $\mathcal{X}_f\coloneqq \{x \in \mathcal{X}:f_\alpha(x)\neq 0, \alpha \in \Omega\}$.

\begin{theorem} \label{FDS}(\textbf{Functional Continuous Deutsch Uncertainty Principle}) Let $(\Omega, \mu)$,  $(\Delta, \nu)$ be   measure spaces and   $\{f_\alpha\}_{\alpha\in \Omega}$, $\{g_\beta\}_{\beta \in \Delta}$ be  1-bounded continuous Parseval p-frames  for a  Banach space $\mathcal{X}$. Then 
	\begin{align*}
	\frac{1}{(\mu(\Omega)\nu(\Delta))^\frac{1}{p}}	\leq \displaystyle\sup_{y \in \mathcal{X}, \|y\|=1}\left(\sup_{\alpha \in \Omega, \beta \in \Delta}|f_\alpha(y)g_\beta(y)|\right)
	\end{align*}
and 
	\begin{align}\label{FD}
		S_f (x)+S_g (x)\geq -p \log \left(\displaystyle\sup_{y \in \mathcal{X}, \|y\|=1}\left(\sup_{\alpha \in \Omega, \beta \in \Delta}|f_\alpha(y)g_\beta(y)|\right)\right)\geq 0, \quad \forall x \in \mathcal{X}_f \cap \mathcal{X}_g.
	\end{align}
\end{theorem}
\begin{proof}
	Let $z \in \mathcal{X}$ be such that $\|z\|=1$. Then 
	\begin{align*}
		1&=\left(\int\limits_{\Omega}|f_\alpha(z)|^p\,d\mu(\alpha)\right)\left(\int\limits_{\Delta}|g_\beta(z)|^p\,d\nu(\beta)\right)=\int\limits_{\Omega}\int\limits_{\Delta}
	|f_\alpha(z)g_\beta(z)|^p\,d\nu(\beta)\,d\mu(\alpha)\\
	&\leq \int\limits_{\Omega}\int\limits_{\Delta}\left(\displaystyle\sup_{y \in \mathcal{X}, \|y\|=1}\left(\sup_{\alpha \in \Omega, \beta \in \Delta}|f_\alpha(y)g_\beta(y)|\right)\right)^p\,d\nu(\beta)\,d\mu(\alpha)\\
	&=\left(\displaystyle\sup_{y \in \mathcal{X}, \|y\|=1}\left(\sup_{\alpha \in \Omega, \beta \in \Delta}|f_\alpha(y)g_\beta(y)|\right)\right)^p\mu(\Omega)\nu(\Delta)
\end{align*}
which gives 
\begin{align*}
	\frac{1}{\mu(\Omega)\nu(\Delta)}\leq \left(\displaystyle\sup_{y \in \mathcal{X}, \|y\|=1}\left(\sup_{\alpha \in \Omega, \beta \in \Delta}|f_\alpha(y)g_\beta(y)|\right)\right)^p.
\end{align*}
Let $x \in  \mathcal{X}_f \cap \mathcal{X}_g$. Then 
\begin{align*}
	S_f (x)+S_g (x)&=	-\int\limits_{\Omega}\int\limits_{\Delta}\left|f_\alpha\left(\frac{x}{\|x\|}\right)\right|^p\left|g_\beta\left(\frac{x}{\|x\|}\right)\right|^p\left[\log \left|f_\alpha\left(\frac{x}{\|x\|}\right)\right|^p+\log \left|g_\beta\left(\frac{x}{\|x\|}\right)\right|^p\right]\,d\nu(\beta)\,d\mu(\alpha)\\
&=-\int\limits_{\Omega}\int\limits_{\Delta}\left|f_\alpha\left(\frac{x}{\|x\|}\right)\right|^p\left|g_\beta\left(\frac{x}{\|x\|}\right)\right|^p\log \left|f_\alpha\left(\frac{x}{\|x\|}\right)g_\beta\left(\frac{x}{\|x\|}\right)\right|^p\,d\nu(\beta)\,d\mu(\alpha)\\
&=-p\int\limits_{\Omega}\int\limits_{\Delta}\left|f_\alpha\left(\frac{x}{\|x\|}\right)\right|^p\left|g_\beta\left(\frac{x}{\|x\|}\right)\right|^p\log \left|f_\alpha\left(\frac{x}{\|x\|}\right)g_\beta\left(\frac{x}{\|x\|}\right)\right|\,d\nu(\beta)\,d\mu(\alpha)\\
&\geq -p\int\limits_{\Omega}\int\limits_{\Delta}\left|f_\alpha\left(\frac{x}{\|x\|}\right)\right|^p\left|g_\beta\left(\frac{x}{\|x\|}\right)\right|^p\log\left(\displaystyle\sup_{y \in \mathcal{X}, \|y\|=1}\left(\sup_{\alpha \in \Omega, \beta \in \Delta}|f_\alpha(y)g_\beta(y)|\right)\right)\,d\nu(\beta)\,d\mu(\alpha)\\
&=-p\log\left(\displaystyle\sup_{y \in \mathcal{X}, \|y\|=1}\left(\sup_{\alpha \in \Omega, \beta \in \Delta}|f_\alpha(y)g_\beta(y)|\right)\right)\int\limits_{\Omega}\int\limits_{\Delta}\left|f_\alpha\left(\frac{x}{\|x\|}\right)\right|^p\left|g_\beta\left(\frac{x}{\|x\|}\right)\right|^p\,d\nu(\beta)\,d\mu(\alpha)\\
&=-p\log\left(\displaystyle\sup_{y \in \mathcal{X}, \|y\|=1}\left(\sup_{\alpha \in \Omega, \beta \in \Delta}|f_\alpha(y)g_\beta(y)|\right)\right).
\end{align*}
\end{proof}
\begin{corollary}
	Theorem \ref{DU} follows from Theorem \ref{FDS}.
\end{corollary}
\begin{proof}
	Let $(\Omega, \mu)$,  $(\Delta, \nu)$ be   measure spaces  and $\{\tau_\alpha\}_{\alpha\in \Omega}$, $\{\omega_\beta\}_{\beta \in \Delta}$ be  	1-bounded continuous Parseval frames for a  Hilbert  space $\mathcal{H}$. Define 
	\begin{align*}
		&f_\alpha:\mathcal{H} \ni h \mapsto \langle h, \tau_\alpha \rangle \in \mathbb{K}; \quad \forall \alpha \in \Omega, \\
		& g_\beta:\mathcal{H} \ni h \mapsto \langle h, \omega_\beta \rangle \in \mathbb{K}, \quad \forall \beta \in \Delta.
	\end{align*}
Now by using Buzano inequality \cite{BUZANO,  FFUJIIKUBO} we get 
\begin{align*}
	\displaystyle\sup_{h \in \mathcal{H}, \|h\|=1}\left(\sup_{\alpha \in \Omega,\beta \in \Delta  }|f_\alpha(h)g_\beta(h)|\right)&=\displaystyle\sup_{h \in \mathcal{H}, \|h\|=1}\left(\sup_{\alpha \in \Omega,\beta \in \Delta }|\langle h, \tau_\alpha \rangle||\langle h, \omega_\beta \rangle|\right)\\
	&\leq \displaystyle\sup_{h \in \mathcal{H}, \|h\|=1}\left(\sup_{\alpha \in \Omega,\beta \in \Delta  }\left(\|h\|^2\frac{\|\tau_\alpha\|\|\omega_\beta\|+|\langle \tau_\alpha, \omega_\beta \rangle |}{2}\right)\right)\\
	&=\frac{1+\displaystyle\sup_{\alpha \in \Omega,\beta \in \Delta }|\langle \tau_j, \omega_k \rangle |}{2}.
\end{align*}
\end{proof}
Theorem  \ref{FDS}  brings the following question.
\begin{question}
 Let $(\Omega, \mu)$,  $(\Delta, \nu)$ be   measure spaces,    $\mathcal{X}$ be a Banach space and $p>1$. For which pairs of continuous Parseval  p-frames $\{f_\alpha\}_{\alpha\in \Omega}, $  and   $\{g_\beta\}_{\beta\in \Delta}$ for  $\mathcal{X}$, we have equality in Inequality (\ref{FD})?
\end{question}
Next we derive a dual inequality of (\ref{FD}). For this we need dual of Definition \ref{A}.
\begin{definition}
	Let 	$(\Omega, \mu)$ be a measure space and $\mathcal{X}$  be a   Banach space over $\mathbb{K}$. A collection $\{\tau_\alpha\}_{\alpha\in \Omega}$ in  $\mathcal{X}$  is said to be a \textbf{continuous Parseval p-frame} ($1\leq p <\infty$) for $\mathcal{X}^*$  if the following conditions  hold.
	\begin{enumerate}[\upshape(i)]
		\item For each $\phi \in \mathcal{X}^*$, the map $\Omega \ni \alpha \mapsto \phi(\tau_\alpha)\in \mathbb{K}$ is measurable.
		\item 
		\begin{align*}
			\|f\|^p=\int_{\Omega}|f(\tau_\alpha)|^p\,d\mu(\alpha), \quad \forall f \in \mathcal{X}^*.
		\end{align*}
	\end{enumerate}
	If $\|\tau_\alpha\|\leq 1$, $\forall \alpha \in \Omega$, then we say that the frame $\{\tau_\alpha\}_{\alpha\in \Omega}$ is 1-bounded. 	
\end{definition}

Given a 1-bounded continuous Parseval p-frame $\{\tau_\alpha\}_{\alpha\in \Omega}$ for $\mathcal{X}^*$,   we define the \textbf{continuous p-Shannon entropy} at a point $f \in \mathcal{X}_\tau^*$ as 
\begin{align*}
	S_f(x)\coloneqq  -\int\limits_{\Omega}\left|\frac{f(\tau_\alpha)}{\|f\|}\right|^p\log \left|\frac{f(\tau_\alpha)}{\|f\|}\right|^p\,d\mu(\alpha)\geq 0, 
\end{align*}
where $\mathcal{X}_\tau^*\coloneqq \{f \in \mathcal{X}: f(\tau_\alpha)\neq 0, \alpha \in \Omega\}\}$.
We now have the following dual to Theorem \ref{FDS}.
\begin{theorem}\label{DUALFDS}
	(\textbf{Continuous Deutsch Uncertainty Principle for Banach spaces}) 
	Let $(\Omega, \mu)$,  $(\Delta, \nu)$ be   measure spaces and   $\{\tau_\alpha\}_{\alpha\in \Omega}$, $\{\omega_\beta\}_{\beta \in \Delta}$ be  1-bounded continuous Parseval p-frames  for the dual $\mathcal{X}^*$ of a  Banach space $\mathcal{X}$. Then 
	\begin{align*}
		\frac{1}{(\mu(\Omega)\nu(\Delta))^\frac{1}{p}}	\leq \displaystyle\sup_{g \in \mathcal{X}^*, \|g\|=1}\left(\sup_{\alpha \in \Omega, \beta \in \Delta}|g(\tau_\alpha)g(\omega_\beta)|\right)
	\end{align*}
	and 
	\begin{align}\label{DFDS}
		S_\tau (f)+S_\omega (f)\geq -p \log \left(\displaystyle\sup_{g \in \mathcal{X}^*, \|g\|=1}\left(\sup_{\alpha \in \Omega, \beta \in \Delta}|g(\tau_\alpha)g(\omega_\beta)|\right)\right)\geq 0, \quad \forall f \in \mathcal{X}_\tau^*\cap \mathcal{X}_\omega^*.
	\end{align}
\end{theorem}
\begin{proof}
	Let $h \in \mathcal{X}^*$ be such that $\|h\|=1$. Then 
\begin{align*}
	1&=\left(\int\limits_{\Omega}|h(\tau_\alpha)|^p\,d\mu(\alpha)\right)\left(\int\limits_{\Delta}h(\omega_\beta)|^p\,d\nu(\beta)\right)=\int\limits_{\Omega}\int\limits_{\Delta}
	|h(\tau_\alpha)h(\omega_\beta)|^p\,d\nu(\beta)\,d\mu(\alpha)\\
	&\leq \int\limits_{\Omega}\int\limits_{\Delta}\left(\displaystyle\sup_{g \in \mathcal{X}^*, \|g\|=1}\left(\sup_{\alpha \in \Omega, \beta \in \Delta}|g(\tau_\alpha)g(\omega_\beta)|\right)\right)^p\,d\nu(\beta)\,d\mu(\alpha)\\
	&=\left(\displaystyle\sup_{g \in \mathcal{X}^*, \|g\|=1}\left(\sup_{\alpha \in \Omega, \beta \in \Delta }|g(\tau_\alpha)g(\omega_\beta)|\right)\right)^p\mu(\Omega)\nu(\Delta)
\end{align*}	
which gives 
\begin{align*}
	\frac{1}{\mu(\Omega)\nu(\Delta)}\leq \left(\displaystyle\sup_{g \in \mathcal{X}^*, \|g\|=1}\left(\sup_{\alpha \in \Omega, \beta \in \Delta}|g(\tau_\alpha)g(\omega_\beta)|\right)\right)^p.
\end{align*}
Let $f \in \mathcal{X}_\tau^*\cap \mathcal{X}_\omega^*$. Then 	
\begin{align*}
		S_\tau (f)+S_\omega (f)&=-\int\limits_{\Omega}\int\limits_{\Delta}\left|\frac{f(\tau_\alpha)}{\|f\|}\right|^p\left|\frac{f(\omega_\beta)}{\|f\|}\right|^p\left[\log \left|\frac{f(\tau_\alpha)}{\|f\|}\right|^p+\log \left|\frac{f(\omega_\beta)}{\|f\|}\right|^p\right]\,d\nu(\beta)\,d\mu(\alpha)\\
		&=-\int\limits_{\Omega}\int\limits_{\Delta}\left|\frac{f(\tau_\alpha)}{\|f\|}\right|^p\left|\frac{f(\omega_\beta)}{\|f\|}\right|^p\log \left|\frac{f(\tau_\alpha)}{\|f\|}\frac{f(\omega_\beta)}{\|f\|}\right|^p\,d\nu(\beta)\,d\mu(\alpha)\\
		&=-p\int\limits_{\Omega}\int\limits_{\Delta}\left|\frac{f(\tau_\alpha)}{\|f\|}\right|^p\left|\frac{f(\omega_\beta)}{\|f\|}\right|^p\log \left|\frac{f(\tau_\alpha)}{\|f\|}\frac{f(\omega_\beta)}{\|f\|}\right|\,d\nu(\beta)\,d\mu(\alpha)\\
		&\geq -p \int\limits_{\Omega}\int\limits_{\Delta}\left|\frac{f(\tau_\alpha)}{\|f\|}\right|^p\left|\frac{f(\omega_\beta)}{\|f\|}\right|^p\log \left(\displaystyle\sup_{g \in \mathcal{X}^*, \|g\|=1}\left(\sup_{\alpha \in \Omega, \beta \in \Delta}|g(\tau_\alpha)g(\omega_\beta)|\right)\right)\,d\nu(\beta)\,d\mu(\alpha)\\
		&=-p\log \left(\displaystyle\sup_{g \in \mathcal{X}^*, \|g\|=1}\left(\sup_{\alpha \in \Omega, \beta \in \Delta}|g(\tau_\alpha)g(\omega_\beta)|\right)\right)\int\limits_{\Omega}\int\limits_{\Delta}\left|\frac{f(\tau_\alpha)}{\|f\|}\right|^p\left|\frac{f(\omega_\beta)}{\|f\|}\right|^p\,d\nu(\beta)\,d\mu(\alpha)\\
		&=-p\log \left(\displaystyle\sup_{g \in \mathcal{X}^*, \|g\|=1}\left(\sup_{\alpha \in \Omega, \beta \in \Delta}|g(\tau_\alpha)g(\omega_\beta)|\right)\right).
\end{align*}
\end{proof}
Theorem  \ref{DUALFDS}  again gives the following question.
\begin{question}
	 Let $(\Omega, \mu)$,  $(\Delta, \nu)$ be   measure spaces,    $\mathcal{X}$ be a Banach space and $p>1$. For which pairs of continuous Parseval  p-frames $\{\tau_\alpha\}_{\alpha\in \Omega}, $  and   $\{\omega_\beta\}_{\beta\in \Delta}$ for  $\mathcal{X}^*$, we have equality in Inequality (\ref{DFDS})
\end{question}

 \bibliographystyle{plain}
 \bibliography{reference.bib}

\end{document}